\documentclass[reqno]{amsart}

\newtheorem{theorem}{Theorem}
\newtheorem{lemma}[theorem]{Lemma}
\newtheorem{proposition}[theorem]{Proposition}
\newtheorem{corollary}[theorem]{Corollary}

\theoremstyle{definition}
\newtheorem{definition}[theorem]{Definition}
\usepackage{graphicx}
\usepackage[all]{xy}
\usepackage{pstricks, enumerate, pst-node, pst-text, pst-plot}

\theoremstyle{remark}
\newtheorem{remark}[theorem]{\bf Remark}

\numberwithin{equation}{section}
\numberwithin{theorem}{section}

\newcommand{\intav}[1]{\mathchoice {\mathop{\vrule width 6pt height 3 pt depth  -2.5pt
\kern -8pt \intop}\nolimits_{\kern -6pt#1}} {\mathop{\vrule width
5pt height 3  pt depth -2.6pt \kern -6pt \intop}\nolimits_{#1}}
{\mathop{\vrule width 5pt height 3 pt depth -2.6pt \kern -6pt
\intop}\nolimits_{#1}} {\mathop{\vrule width 5pt height 3 pt depth
-2.6pt \kern -6pt \intop}\nolimits_{#1}}}

\newcommand{\intavl}[1]{\mathchoice {\mathop{\vrule width 6pt height 3 pt depth  -2.5pt
\kern -8pt \intop}\limits_{\kern -6pt#1}} {\mathop{\vrule width 5pt
height 3  pt depth -2.6pt \kern -6pt \intop}\nolimits_{#1}}
{\mathop{\vrule width 5pt height 3 pt depth -2.6pt \kern -6pt
\intop}\nolimits_{#1}} {\mathop{\vrule width 5pt height 3 pt depth
-2.6pt \kern -6pt \intop}\nolimits_{#1}}}

\newcommand{\mc}{\mathcal}

\newcommand{\cc}{\mc{C}}
\newcommand{\D}{\mc{D}}

\newcommand{\U}{\mc{U}}

\newcommand{\R}{\mathbb{R}}
\newcommand{\N}{\mathbb{N}}

\newcommand{\Z}{\mathbb{Z}}

\newcommand{\la}{\lambda}
\newcommand{\te}{{\theta}}
\newcommand{\Te}{{\Theta}}

\newcommand{\ve}{{\varepsilon}}
\newcommand{\vr}{{\varphi}}

\def\proj{{\rm proj}}
\def\hd{{\rm HD}}

\begin{document}

\title[Yet another proof of Marstrand's Theorem]{Yet another proof of Marstrand's Theorem}

\author[Yuri Lima]{Yuri Lima}
\address{Instituto Nacional de Matem\'atica Pura e Aplicada, Estrada Dona Castorina 110, 22460-320, Rio de Janeiro, Brasil.}
\email{yurilima@impa.br}

\author[Carlos Gustavo Moreira]{Carlos Gustavo Moreira}
\address{Instituto Nacional de Matem\'atica Pura e Aplicada, Estrada Dona Castorina 110, 22460-320, Rio de Janeiro, Brasil.}
\email{gugu@impa.br}
\


\date{January, 17, 2011}


\begin{abstract}
In a paper from 1954 Marstrand proved that if $K\subset\R^2$ is a Borel set with Hausdorff dimension greater than
$1$, then its one-dimensional projection has positive Lebesgue measure for almost-all directions. In this article, we
give a combinatorial proof of this theorem, extending the techniques developed in our previous paper \cite{LM2}.
\end{abstract}

\maketitle

\section{Introduction}\label{sec introduction}

If $U$ is a subset of $\R^n$, the diameter of $U$ is $|U|=\sup\{|x-y|;x,y\in U\}$
and, if $\mathcal U$ is a family of subsets of $\R^n$, the {\it diameter} of $\mathcal U$ is defined as
$$\left\|\mathcal U\right\|=\sup_{U\in\,\mathcal U}|U|.$$
Given $s>0$, the {\it Hausdorff $s$-measure} of a set $K\subset\R^n$ is
$$m_s(K)=\lim_{\ve\rightarrow 0}\left(\inf_{\mathcal U\text{ covers }K\atop{\left\|\mathcal U\right\|<\,\ve}}
\sum_{U\in\,\mathcal U}|U|^s\right).$$
In particular, when $n=1$, $m=m_1$ is the Lebesgue measure of Lebesgue measurable sets on $\R$.
It is not difficult to show that there exists a unique $s_0\ge0$ for which $m_s(K)=\infty$ if $s<s_0$ and
$m_s(K)=0$ if $s>s_0$. We define the Hausdorff dimension of $K$ as $\hd(K)=s_0$. Also, for each $\te\in\R$,
let $v_\te=(\cos \te,\sin\te)$, $L_\te$ the line in $\R^2$ through the origin containing $v_\te$ and
$\proj_\te:\R^2\rightarrow L_\te$ the orthogonal projection. From now on, we'll restrict $\te$ to
the interval $[-\pi/2,\pi/2]$, because $L_{\te}=L_{\te+\pi}$.

In $1954$, J. M. Marstrand \cite{Ma} proved the following result on the fractal dimension of plane sets.

\begin{theorem}\label{thm 1}If $K\subset\R^2$ is a Borel set such that $\hd(K)>1$, then
$m(\proj_\te(K))>0$ for $m$-almost every $\te\in\R$.
\end{theorem}

The proof is based on a qualitative characterization of the ``bad" angles
$\te$ for which the result is not true. Specifically, Marstrand exhibits a Borel measurable function
$f(x,\te)$, $(x,\te)\in\R^2\times[-\pi/2,\pi/2]$, such that $f(x,\te)=\infty$ for $m_d$-almost every $x\in K$,
for every ``bad" angle. In particular,
\begin{equation}\label{equacao 1}
\int_K f(x,\te)dm_d(x)=\infty.
\end{equation}
On the other hand, using a version of Fubini's Theorem, he proves that
$$\int_{-\pi/2}^{\pi/2}d\te\int_K f(x,\te)dm_d(x)=0$$
which, in view of (\ref{equacao 1}), implies that
$$m\left(\{\te\in[-\pi/2,\pi/2]\,;\, m(\proj_\te(K))=0\}\right)=0.$$
These results are based on the analysis of rectangular densities of points.

Many generalizations and simpler proofs have appeared since. One of them came in 1968 by R. Kaufman
who gave a very short proof of Marstrand's Theorem using methods of potential theory. See \cite{K}
for his original proof and \cite{M1}, \cite{PT} for further discussion.

In this article, we give a new proof of Theorem \ref{thm 1}. Our proof makes a study on the fibers
$K\cap{\proj_\te}^{-1}(v)$, $(\te,v)\in\R\times L_\te$, and relies on two facts:

\noindent (I) Transversality condition: given two squares on the plane, the Lebesgue measure of the set of angles
for which their projections have nonempty intersection has an upper bound. See Subsection
\ref{sub transversality}.

\noindent (II) After a regularization of $K$, (I) enables us to conclude that, except for a small set of
angles $\te\in\R$, the fibers $K\cap {\proj_\te}^{-1}(v)$ are not concentrated in a thin region. As a consequence,
$K$ projects into a set of positive Lebesgue measure.

The idea of (II) is based on the work \cite{Mo1} of the second author and was employed in \cite{LM2} to
develop a combinatorial proof of Theorem \ref{thm 1} when $K$ is the product of two regular Cantor sets.
In the present paper, we give a combinatorial proof of Theorem \ref{thm 1} without any restrictions on $K$.
Compared to other proofs of Marstrand's Theorem, the new ideas here are the discretization of the argument and
the use of dyadic covers, which allow the simplification of the method employed.

We also show that the push-forward measure of the restriction of $m_d$ to $K$,
defined as $\mu_\te=(\proj_\te)_*(m_d|_K)$, is absolutely continuous with respect to $m$,
for $m$-almost every $\te\in\R$, and its Radon-Nykodim derivative is square-integrable.

\begin{theorem}\label{thm 2}
The measure $\mu_\te$ is absolutely continuous with respect to $m$ and
its Radon-Nykodim derivative is an $L^2$ function, for $m$-almost every $\te\in\R$.
\end{theorem}

\begin{remark}
Theorem \ref{thm 2}, as in this work, follows from most proofs of Marstrand's theorem
and, in particular, is not new as well.
\end{remark}

Marstrand's Theorem is a classical result in Geometric Measure Theory. In particular, if $K=K_1\times K_2$ is a cartesian product of two one-dimensional subsets of $\R$, Marstrand's theorem translates to ``$m(K_1+\la K_2)>0$
for $m$-almost every $\la\in\R$''. The investigation of such {\it arithmetic sums} $K_1+\la K_2$ has been an
active area of Mathematics, in special when $K_1$ and $K_2$ are dynamically defined Cantor sets. Although he did
not know, M. Hall \cite{Ha} proved, in 1947, that the Lagrange spectrum\footnote{The Lagrange spectrum is the set of
best constants of rational approximations of irrational numbers. See \cite{CuF} for the specific description.}
contains a whole half line, by showing that the arithmetic sum $K(4)+K(4)$ of a certain Cantor set $K(4)\subset\R$
with itself contains $[6,\infty)$.

Marstrand's Theorem for product of Cantor sets is also fundamental in certain results of dynamical bifurcations,
namely homoclinic bifurcations in surfaces. For instance, in \cite{PY} it is used to show that hyperbolicity is
not prevalent in homoclinic bifurcations associated to horseshoes with Hausdorff dimension larger than one;
in \cite{MY} it is used to prove that stable intersections of regular Cantor sets are dense in the region
where the sum of their Hausdorff dimensions is larger than one; in \cite{MY2} to show that, for homoclinic
bifurcations associated to horseshoes with Hausdorff dimension larger than one, typically there are open sets
of parameters with positive Lebesgue density at the initial bifurcation parameter corresponding to persistent
homoclinic tangencies.

In the connection of these two applications, we point out that a formula for the Hausdorff
dimension of $K_1+K_2$, under mild assumptions of non-linear Cantor sets $K_1$ and $K_2$, has been
obtained by the second author in \cite{Mo1} and applied in \cite{Mo2} to prove that the Hausdorff dimension of
the Lagrange spectrum increases continuously. In parallel to this non-linear setup, Y. Peres and P. Shmerkin
proved the same phenomena happen to self-similar Cantor sets without algebraic resonance \cite{PeSh}.
Finally, M. Hochman and P. Shmerkin extended and unified many results concerning projections of products of
self-similar measures on regular Cantor sets \cite{HoSh}.

The paper is organized as follows. In Section \ref{sec preliminaries} we introduce the basic notations
and definitions. Section \ref{sec calculations} is devoted to the main calculations, including the
transversality condition in Subsection \ref{sub transversality} and the proof of existence of good dyadic
covers in Subsection \ref{subsec dyadic}. Finally, in Section \ref{sec proof of thm}
we prove Theorems \ref{thm 1} and \ref{thm 2}. We also collect final remarks
in Section \ref{sec final remarks}.

\section{Preliminaries}\label{sec preliminaries}

\subsection{Notation}\label{subsec notation}

The distance in $\R^2$ will be denoted by $|\cdot|$. Let $B_r(x)$ denote the open ball of $\R^2$ centered in $x$
with radius $r$. As in Section \ref{sec introduction}, the diameter of $U\subset\R^2$
is $|U|=\sup\{|x-y|;x,y\in U\}$ and, if $\mathcal U$ is a family of subsets of $\R^2$, the diameter of
$\mathcal U$ is defined as
$$\left\|\mathcal U\right\|=\sup_{U\in\,\mathcal U}|U|.$$
Given $s>0$, the Hausdorff $s$-measure of a set $K\subset\R^2$ is $m_s(K)$ and its Hausdorff dimension is
$\hd(K)$. In this work, we assume $K$ is contained in $[0,1)^2$.

\begin{definition}\label{def s-set}
A Borel set $K\subset\R^2$ is an $s$-{\it set} if $\hd(K)=s$ and $0<m_s(K)<\infty$.
\end{definition}

Let $m$ be the Lebesgue measure of Lebesgue measurable sets on $\R$.
For each $\te\in\R$, let $v_\te=(\cos \te,\sin\te)$, $L_\te$ the line in $\R^2$ through the origin containing
$v_\te$ and $\proj_\te:\R^2\rightarrow L_\te$ the orthogonal projection onto $L_\te$.

A square $[a,a+l)\times[b,b+l)\subset\R^2$ will be denoted by $Q$ and its center, the point
$(a+l/2,b+l/2)$, by $x$.

We use Vinogradov notation to compare the asymptotic of functions.

\begin{definition}\label{def vinogradov}
Let $f,g:\N$ or $\R\rightarrow\R$ be two real-valued functions. We say $f\ll g$ if there is a
constant $C>0$ such that
$$|f(x)|\le C\cdot |g(x)|\,,\ \ \forall\,x\in\N\text{ or }\R.$$
If $f\ll g$ and $g\ll f$, we write $f\asymp g$.
\end{definition}

\subsection{Dyadic squares}\label{subsec dyadic}

Let $\D_0$ be the family of unity squares of $\R^2$ congruent to $[0,1)^2$ and with vertices in the lattice
$\Z^2$. Dilating this family by a factor of $2^{-i}$, we obtain the family $\D_i$, $i\in\Z$.

\begin{definition}
Let $\D$ denote the union of $\D_i$, $i\in\Z$. A {\it dyadic square} is any element $Q\in\D$.
\end{definition}

The dyadic squares possess the following properties:
\begin{enumerate}[(1)]
\item Every $x\in\R^2$ belongs to exactly one element of each family $\D_i$.
\item Two dyadic squares are either disjoint or one is contained in the other.
\item A dyadic square of $\D_i$ is contained in exactly one dyadic square of $\D_{i-1}$ and contains
exactly four dyadic squares of $\D_{i+1}$.
\item Given any subset $U\subset\R^2$, there are four dyadic squares, each with side length at most $2\cdot|U|$,
whose union contains $U$.
\end{enumerate}

(1) to (3) are direct. To prove (4), let $R$ be smallest rectangle of $\R^2$ with sides parallel to
the axis that contains $\overline{U}$. The sides of $R$ have length at most $|U|$. Let $i\in\Z$ such that
$2^{-i-1}\le |U|<2^{-i}$ and choose a dyadic square $Q\in\D_i$ that intersects $R$. If $Q$ contains $U$, we're
done. If not, $Q$ and three of its neighbors cover $U$.

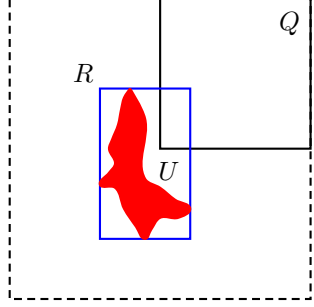
\begin{center}
\psset{unit=.4cm} \begin{pspicture}(0,0)(10,10)

\pspolygon[linestyle=dashed,dash=3pt 2pt](0,0)(10,0)(10,10)(0,10)\pspolygon(5,5)(5,10)(10,10)(10,5)
\pspolygon[linecolor=blue](6,7)(6,2)(3,2)(3,7)

\pscustom[fillstyle=solid,fillcolor=red,linecolor=red]{
\pscurve(4,7)(4.5,6)(4.5,4)(5,3.7)(5.5,3.3)(6,3)(5.6,2.7)(5,2.7)(4.7,2.3)(4.5,2)(4.2,2.4)(3.8,3)
(3.5,3.7)(3,3.8)(3.3,4.2)(3.5,4.5)(3.3,5)(3.4,5.3)(3.6,5.8)(3.7,6.3)(3.8,6.7)(4,7)
}

\uput[270](9.3,10){$Q$}\uput[135](3.1,6.9){$R$}\uput[270](5.3,5){$U$}
\end{pspicture}\end{center}

\begin{definition}
A {\it dyadic cover} of $K$ is a finite subset $\cc\subset\D$ of disjoint dyadic squares such that
$$K\subset\bigcup_{Q\in\cc}Q.$$
\end{definition}

Due to (4), for any family $\U$ of subsets of $\R^2$, there is a dyadic family $\cc$ such that
$$\bigcup_{Q\in\cc}Q\supset\bigcup_{U\in\U}U\ \ \ \text{ and}\ \ \ \sum_{Q\in\cc}|Q|^s<16\cdot \sum_{U\in\U}|U|^s$$
and so, if $K$ is an $s$-set, there exists a sequence $(\cc_i)_{i\ge 1}$ of dyadic covers of $K$ such that
\begin{equation}\label{equation 4}
\sum_{Q\in\cc_i}|Q|^s\asymp 1\,.
\end{equation}

\section{Calculations}\label{sec calculations}

Let $K\subset\R^2$ be a Borel set with Hausdorff dimension greater than one. From now on, we assume that every
cover of $K$ is composed of dyadic squares of sides at most one. Before going into the calculations, we make
the following reduction.

\begin{lemma}\label{lemma 1}
Let $K$ be a Borel subset of $\R^2$. Given $s<\hd(K)$, there exists an $s$-set $K'\subset K$ such that
$$m_s(K'\cap B_r(x))\ll r^d\,,\ \ \text{for any }x\in\R^2\text{ and }0<r\le 1.$$
In other words, there exists a constant $b>0$ such that
$$m_s(K'\cap B_r(x))\le b\cdot r^d\,,\ \ \text{for any }x\in\R^2\text{ and }0<r\le 1.$$
\end{lemma}

By the above lemma, we may assume $K$ is an $s$-set, with $s>1$.

Given a dyadic cover $\cc$ of $K$, let, for each $\te\in[-\pi/2,\pi/2]$, $f_\te^\cc:L_\te\rightarrow\R$ be the
function defined by
$$f_\te^\cc(x)=\sum_{Q\in\cc}\chi_{\proj_\te(Q)}(x)\cdot |Q|^{s-1}\,,$$
where $\chi_{\proj_\te(Q)}$ denotes the characteristic function of the set $\proj_\te(Q)$. The reason we consider
this function is that it captures the Hausdorff s-measure of $K$ in the sense that
\begin{eqnarray*}
\int_{L_\te}f_\te^\cc(x)dm(x)&=&\sum_{Q\in\cc}|Q|^{s-1}\cdot\int_{L_\te}\chi_{\proj_\te(Q)}(x)dm(x)\\
                             &=&\sum_{Q\in\cc}|Q|^{s-1}\cdot m(\proj_\te(Q))
\end{eqnarray*}
which, as $|Q|/2\le m(\proj_\te(Q))\le |Q|$, satisfies
$$\int_{L_\te}f_\te^\cc(x)dm(x)\asymp\sum_{Q\in\cc}|Q|^s.$$
If in addition $\cc$ satisfies (\ref{equation 4}), then
\begin{equation}\label{equation 5}
\int_{L_\te}f_\te^\cc(x)dm(x)\asymp 1\ ,\ \ \forall\,\te\in[-\pi/2,\pi/2].
\end{equation}
Denoting the union $\bigcup_{Q\in\cc}Q$ by $\cc$ as well, an application of the Cauchy-Schwarz inequality gives that
$$m(\proj_\te(\cc))\cdot\left(\int_{\proj_\te(\cc)}\left(f_\te^\cc\right)^2dm\right)
\ge\left(\int_{\proj_\te(\cc)}f_\te^\cc dm\right)^2\asymp 1.$$
The above inequality implies that if $(\cc_i)_{i\ge 1}$ is a sequence of dyadic covers of $K$  satisfying
(\ref{equation 4}) with diameters converging to zero and the $L^2$-norm of $f_\te^{\cc_i}$ is uniformly bounded,
that is
\begin{equation}\label{equation 6}
\int_{\proj_\te(\cc_i)}\left(f_\te^{\cc_i}\right)^2dm\ll 1,
\end{equation}
then
$$m(\proj_\te(K))=\lim_{i\rightarrow\infty}m(\proj_\te(\cc_i))\gg 1$$
and so $\proj_\te(K)$ has positive Lebesgue measure, as wished. This conclusion will be obtained for
$m$-almost every $\te\in[-\pi/2,\pi/2]$ by showing that
\begin{equation}\label{equation 2}
I_i\doteq\int_{-\pi/2}^{\pi/2}d\te\int_{L_\te}\left(f_\te^{\cc_i}\right)^2dm\ll 1.
\end{equation}

\subsection{Rewriting the integral $I_i$}\label{sub rewriting}

For simplicity, let $f$ denote $f_\te^{\cc_i}$. Then the interior integral of (\ref{equation 2}) becomes

\begin{eqnarray*}
\int_{L_\te}f^2dm&=&\int_{L_\te}\left(\sum_{Q\in\cc_i}\chi_{\proj_\te(Q)}\cdot |Q|^{s-1}\right)\cdot
\left(\sum_{\tilde Q\in\cc_i}\chi_{\proj_\te(\tilde Q)}\cdot |\tilde Q|^{s-1}\right)dm\\
&=&\sum_{Q,\tilde Q\in\cc_i}|Q|^{s-1}\cdot|\tilde Q|^{s-1}\cdot\int_{L_\te}\chi_{\proj_\te(Q)\cap\proj_\te(\tilde Q)}dm\\
&=&\sum_{Q,\tilde Q\in\cc_i}|Q|^{s-1}\cdot|\tilde Q|^{s-1}\cdot m(\proj_\te(Q)\cap\proj_\te(\tilde Q))
\end{eqnarray*}
and, using the inequalities
$$m(\proj_\te(Q)\cap\proj_\te(\tilde Q))\le \min\{m(\proj_\te(Q)),m(\proj_\te(\tilde Q)\}\le\min\{|Q|,|\tilde Q|\}\,,$$
it follows that
\begin{equation}\label{equation 1}
\int_{L_\te}f^2dm\ll\sum_{Q,\tilde Q\in\cc_i}|Q|^{s-1}\cdot
|\tilde Q|^{s-1}\cdot\min\{|Q|,|\tilde Q|\}.
\end{equation}
We now proceed to prove (\ref{equation 2}) by a double-counting argument. To this matter, consider, for each
pair of squares $(Q,\tilde Q)\in \cc_i\times\cc_i$, the set
$$\Te_{Q,\tilde Q}=\left\{\te\in[-\pi/2,\pi/2];\proj_\te(Q)\cap\proj_\te(\tilde Q)\not=\emptyset\right\}.$$
Then
\begin{eqnarray}
I_i&\ll&\sum_{Q,\tilde Q\in\cc_i}|Q|^{s-1}\cdot|\tilde Q|^{s-1}\cdot\min\{|Q|,|\tilde Q|\}\cdot \int_{-\pi/2}^{\pi/2}\chi_{\Te_{Q,\tilde Q}}(\te)d\te\nonumber\\
&=&\sum_{Q,\tilde Q\in\cc_i}|Q|^{s-1}\cdot|\tilde Q|^{s-1}\cdot\min\{|Q|,|\tilde Q|\}\cdot m(\Te_{Q,\tilde Q})\,. \label{equation 3}
\end{eqnarray}

\subsection{Transversality condition}\label{sub transversality}

This subsection estimates the Lebesgue measure of $\Te_{Q,\tilde Q}$.

\begin{lemma}\label{lemma 2}
If $Q,\tilde Q$ are squares of $\R^2$ and $x,\tilde x\in\R^2$ are its centers, respectively, then
$$m\left(\Te_{Q,\tilde Q}\right)\le2\pi\cdot\dfrac{\max\{|Q|,|\tilde Q|\}}{|x-\tilde x|}\,\cdot$$
\end{lemma}

\begin{proof} Let $\te\in\Te_{Q,\tilde Q}$ and consider the figure.\\

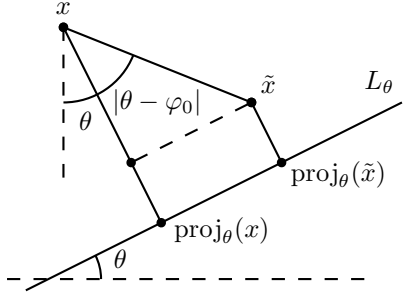
\begin{center}
\psset{unit=.5cm} \begin{pspicture}(-3,-.7)(8,6)

\psline(-2,-1)(8,4)\psline{*-*}(-1,6)(1.6,.8)\psline{*-*}(4,4)(4.8,2.4)\psline{*-*}(-1,6)(4,4)
\psline[linestyle=dashed,dash=5pt 5pt](-2.5,-.7)(7,-.7)\psline[linestyle=dashed,dash=5pt 5pt](-1,2)(-1,6)
\psline[linestyle=dashed,dash=5pt 5pt]{*-*}(.8,2.4)(4,4)\psarc(-1,6){2}{270}{338}\psarc(-1.4,-.7){1.4}{0}{27}

\uput[90](-1,6){$x$}\uput[45](4,4){$\tilde x$}\uput[-45](1.7,1.2){$\proj_\te(x)$}
\uput[-45](4.8,2.7){$\proj_\te(\tilde x)$}\uput[90](7.4,3.9){$L_\te$}\uput[-60](-.7,4.1){$\te$}
\uput[-60](.6,4.7){$|\te-\vr_0|$}\uput[30](0,-.5){$\te$}
\end{pspicture} \end{center}
Since $\proj_\te(Q)$ has diameter at most $|Q|$ (and the same happens
to $\tilde Q$), we have $|\proj_\te(x)-\proj_\te(\tilde x)|\le 2\cdot\max\{|Q|,|\tilde Q|\}$ and then,
by elementary geometry,
\begin{eqnarray*}
\sin(|\te-\vr_0|)&=  &\dfrac{|\proj_\te(x)-\proj_\te(\tilde x)|}{|x-\tilde x|}\\
                 &\le&2\cdot\dfrac{\max\{|Q|,|\tilde Q|\}}{|x-\tilde x|}\\
\Longrightarrow\hspace{1cm}|\te-\vr_0|&\le&\pi\cdot\dfrac{\max\{|Q|,|\tilde Q|\}}{|x-\tilde x|}\,,
\end{eqnarray*}
because $\sin^{-1}y\le \pi y/2$. As $\vr_0$ is fixed, the lemma is proved.
\end{proof}

We point out that, although ingenuous, Lemma \ref{lemma 2} expresses the crucial property of transversality that
makes the proof work, and all results related to Marstrand's Theorem use a similar idea in one way or another.
See \cite{R} where this tranversality condition is also exploited.

By Lemma \ref{lemma 2} and (\ref{equation 3}), we get
\begin{equation}\label{equation 9}
I_i\ll\sum_{Q,\tilde Q\in\cc_i}|x-\tilde x|^{-1}\cdot |Q|^s\cdot|\tilde Q|^s.
\end{equation}

\subsection{Good dyadic covers}\label{sub good covers}

We estimate the last summand by choosing appropriate dyadic covers $\cc_i$. Let $\cc$ be an arbitrary
dyadic cover of $K$.

\begin{definition}\label{def good cover}
The dyadic cover $\cc$ is {\it good} if
\begin{equation}\label{equation 7}
\sum_{\tilde Q\in\cc\atop{\tilde Q\subset Q}}|\tilde Q|^s\ll |Q|^s\,,
\end{equation}
where $Q$ runs over all elements of $\D$.
\end{definition}

The existence of good dyadic covers is provided below.

\begin{proposition}\label{prop good cover}
Let $K\subset\R^2$ be an $s$-set. Then, for any $\delta>0$, there exists a good dyadic cover of
$K$ with diameter less than $\delta$.
\end{proposition}

\begin{proof}
Begin with an arbitrary dyadic cover $\cc$ with diameter less than $\delta$. As $K$ is an $s$-set, we
may assume that (\ref{equation 7}) holds for every $Q\in\bigcup_{\|\cc\|<2^{-i}}\D_i$. To see this, let
$i_0\ge 1$ such that $2^{-i_0-1}<\|\cc\|\le 2^{-i_0}$. Look at the restriction of $\cc$ to each
$Q\in\D_{i_0}$ individually. By Lemma \ref{lemma 1}, we may assume that
$$\sum_{\tilde Q\in\cc\atop{\tilde Q\subset Q}}|\tilde Q|^s\le b\cdot |Q|^s,\ \ \forall\,Q\in\D_{i_0}.$$
As a consequence,
\begin{equation}\label{equation 8}
\sum_{\tilde Q\in\cc\atop{\tilde Q\subset Q}}|\tilde Q|^s\le (4^{i_0}\cdot b)\cdot |Q|^s
\end{equation}
for any $Q\in\bigcup_{0\le i\le i_0}\D_i$, that is, (\ref{equation 7}) holds for large scales. To control
the small ones, apply the following operation: whenever $Q\in\bigcup_{i>i_0}\D_i$ is such that
$$\sum_{\tilde Q\in\cc\atop{\tilde Q\subset Q}}|\tilde Q|^s> |Q|^s,$$
we change $\cc$ by $\cc\cup\{Q\}\backslash\{\tilde Q\in\cc\,;\,\tilde Q\subset Q\}$. It is clear that such
operation preserves the inequality (\ref{equation 8}) and so, after a finite number of steps, we end up with a
good dyadic cover.
\end{proof}

\section{Proof of Theorems 1.1 and 1.2}\label{sec proof of thm}

\begin{proof}[Proof of Theorem \ref{thm 1}.]
Let $(\cc_i)_{i\ge 1}$ be a sequence of good dyadic covers such that $\|\cc_i\|\rightarrow 0$. By
(\ref{equation 9}), we have
\begin{eqnarray*}
I_i&\ll&\sum_{Q,\tilde Q\in\cc_i}|x-\tilde x|^{-1}\cdot |Q|^s\cdot|\tilde Q|^s\\
   &=&\sum_{Q\in\cc_i}\sum_{j=0}^{\infty}\sum_{\tilde Q\in\cc_i\atop{2^{-j-1}<|x-\tilde x|\le 2^{-j}}}
     |x-\tilde x|^{-1}\cdot |Q|^s\cdot|\tilde Q|^s\\
   &\le&\sum_{Q\in\cc_i}\sum_{j=0}^{\infty}\sum_{\tilde Q\in\cc_i\atop{\tilde Q\subset B_{3\cdot 2^{-j}}(x)}}
     |x-\tilde x|^{-1}\cdot |Q|^s\cdot|\tilde Q|^s\\
   &\ll&\sum_{Q\in\cc_i}|Q|^s\sum_{j=0}^{\infty}2^j\cdot \left(2^{-j}\right)^s\\
   &=&\sum_{Q\in\cc_i}|Q|^s\sum_{j=0}^{\infty}\left(2^j\right)^{1-s}\\
   &\ll&\sum_{Q\in\cc_i}|Q|^s\\
   &\ll&1,
\end{eqnarray*}
establishing (\ref{equation 2}). Define, for each $\ve>0$, the sets
$$G^i_\ve=\left\{\te\in[-\pi/2,\pi/2]\,;\,\int_{L_\te}\left(f_\te^{\cc_i}\right)^2dm<\ve^{-1}\right\},\ \ i\ge 1.$$
Then $m\left([-\pi/2,\pi/2]\backslash G_\ve^i\right)\ll\ve$, and the same holds for the set
$$G_\ve=\bigcap_{i\ge 1}\bigcup_{j=i}^{\infty}G_\ve^j\,.$$
If $\te\in G_\ve$, then
$$m\left(\proj_\te(\cc_i)\right)\gg\ve\,,\ \text{ for infinitely many }n,$$
which implies that $m\left(\proj_\te(K)\right)>0$. Finally, the set $G=\bigcup_{i\ge 1}G_{1/i}$ satisfies $m([-\pi/2,\pi/2]\backslash G)=0$ and $m\left(\proj_\te(K)\right)>0$, for any $\te\in G$.
\end{proof}

A direct consequence is the

\begin{corollary}\label{corollary}
The measure $\mu_\te=(\proj_\te)_*(m_d|_K)$ is absolutely continuous with respect to $m$, for
$m$-almost every $\te\in\R$.
\end{corollary}

\begin{proof}
By Theorem \ref{thm 1}, we have the implication
\begin{equation}\label{equation 10}
X\subset K\, ,\ m_d(X)>0\ \ \Longrightarrow\ \ m(\proj_\te(X))>0,\ \ m\text{-almost every }\te\in\R,
\end{equation}
which is sufficient for the required absolute continuity. Indeed, if $Y\subset L_\te$ satisfies $m(Y)=0$, then
$$\mu_\te(Y)=m_d(X)=0\,,$$
where $X={\proj_\te}^{-1}(Y)$. Otherwise, by (\ref{equation 10}) we would have $m(Y)=m(\proj_\te(X))>0$,
contradicting the assumption.
\end{proof}

Let $f_\te=d\mu_\te/dm$. By the proof of Theorem \ref{thm 1}, we have
\begin{equation}\label{equation 11}
\left\|f_\te^{\cc_i}\right\|_{L^2}\ll 1,\ \ m\text{-a.e. }\te\in\R.
\end{equation}

\begin{proof}[Proof of Theorem 2.]
Define, for each $\ve>0$, the function $f_{\te,\ve}:L_\te\rightarrow\R$ by
$$f_{\te,\ve}(x)=\dfrac{1}{2\ve}\int_{x-\ve}^{x+\ve}f_\te(y)dm(y),\ \ x\in L_\te.$$
As $f_\te$ is an $L^1$-function, the Lebesgue differentiation theorem gives that
$f_\te(x)=\lim_{\ve\rightarrow 0}f_{\te,\ve}(x)$ for m-almost every $x\in L_\te$. If we manage to show
that\footnote{We consider $\left\|f_{\te,\ve}\right\|_{L^2}$ as a function of $\ve>0$.}
\begin{equation}\label{equation 12}
\left\|f_{\te,\ve}\right\|_{L^2}\ll 1,\ \ m\text{-a.e. }\te\in\R,
\end{equation}
then Fatou's lemma establishes the theorem. To this matter, first observe that
\begin{eqnarray*}
f_{\te,\ve}(x)&=&\dfrac{1}{2\ve}\int_{x-\ve}^{x+\ve}f_\te(y)dm(y)\\
              &=&\dfrac{1}{2\ve}\cdot\mu_\te([x-\ve,x+\ve])\\
              &=&\dfrac{1}{2\ve}\cdot m_d\left((\proj_{\te})^{-1}([x-\ve,x+\ve])\cap K\right).
\end{eqnarray*}
In order to estimate this last term, fix $\ve>0$ and let $i>0$ such that $\cc=\cc_i$ has diameter less
than $\ve$. Then
\begin{eqnarray*}
m_d\left((\proj_{\te})^{-1}([x-\ve,x+\ve])\cap K\right)
  &\le&\sum_{Q\in\cc\atop{\proj_\te(Q)\subset[x-2\ve,x+2\ve]}}m_d(Q\cap K)\\
  &\ll&\sum_{Q\in\cc\atop{\proj_\te(Q)\subset[x-2\ve,x+2\ve]}}|Q|^s\\
  &\asymp&\sum_{Q\in\cc\atop{\proj_\te(Q)\subset[x-2\ve,x+2\ve]}}|Q|^{s-1}\cdot m(\proj_\te(Q))\\
  &\le&\int_{x-2\ve}^{x+2\ve}f_\te^{\cc}(y)dm(y),
\end{eqnarray*}
where in the second inequality we applied the conditions of Lemma \ref{lemma 1}. By the Cauchy-Schwarz
inequality, we obtain
$$\left|f_{\te,\ve}(x)\right|^2\ll\dfrac{1}{2\ve}\int_{x-2\ve}^{x+2\ve}\left|f_\te^\cc(y)\right|^2dm(y)$$
and so
\begin{eqnarray*}
\left\|f_{\te,\ve}\right\|_{L^2}^2
    &\ll&\int_{L_\te}\dfrac{1}{2\ve}\int_{x-2\ve}^{x+2\ve}\left|f_\te^\cc(y)\right|^2dm(y)dm(x)\\
&\asymp&\int_{L_\te}\left|f_\te^\cc\right|^2dm\\
&=&\left\|f_\te^\cc\right\|_{L^2}^2
\end{eqnarray*}
which, by (\ref{equation 11}), establishes (\ref{equation 12}).
\end{proof}

\section{Concluding remarks}\label{sec final remarks}

The good feature of the proof is that the discretization idea may be applied to other contexts. For example,
we prove in \cite{LM} a Marstrand type theorem in an arithmetical context.

\section*{Acknowledgments}

The authors are thankful to IMPA for the excellent ambient during the preparation of this
manuscript. The authors are also grateful to Carlos Matheus for carefully reading the preliminary version of this
work. This work was financially supported by CNPq-Brazil and Faperj-Brazil.

\bibliographystyle{amsplain}

\begin{thebibliography}{99}

\bibitem{CuF}
{\bf T. Cusick and M. Flahive},
\newblock{\it The Markov and Lagrange spectra},
\newblock Mathematical Surveys and Monographs {\bf 30}, AMS.

\bibitem{F}
{\bf K. Falconer},
\newblock{\it The geometry of fractal sets},
\newblock Cambridge Tracts in Mathematics, Cambridge (1986).

\bibitem{Ha}
{\bf M. Hall},
\newblock{\it On the sum and product of continued fractions},
\newblock Annals of Mathematics {\bf 48} (1947), 966--993.

\bibitem{HoSh}
{\bf M. Hochman and P. Shmerkin},
\newblock{\it Local entropy averages and projections of fractal measures},
\newblock to appear in Annals of Mathematics.

\bibitem{K}
{\bf R. Kaufman},
\newblock{\it On Hausdorff dimension of projections},
\newblock Mathematika {\bf 15} (1968), 153--155.

\bibitem{LM}
{\bf Y. Lima and C.G. Moreira},
\newblock{\it A Marstrand theorem for subsets of integers},
\newblock available at http://arxiv.org/abs/1011.0672.

\bibitem{LM2}
{\bf Y. Lima and C.G. Moreira},
\newblock{\it A combinatorial proof of Marstrand's theorem for products of regular Cantor sets},
to appear in Expositiones Mathematicae.

\bibitem{Ma}
{\bf J.M. Marstrand},
\newblock{\it Some fundamental geometrical properties of plane sets of fractional dimensions},
\newblock Proceedings of the London Mathematical Society {\bf 3} (1954), vol. 4, 257--302.

\bibitem{M1}
{\bf P. Mattila},
\newblock{\it Hausdorff dimension, projections, and the Fourier transform},
\newblock  Publ. Mat. {\bf 48} (2004), no. 1, 3--48.


\bibitem{Mo1}
{\bf C.G. Moreira},
\newblock{\it A dimension formula for images of cartesian products of regular Cantor sets by
differentiable real maps},
\newblock in preparation.

\bibitem{Mo2}
{\bf C.G. Moreira},
\newblock{\it Geometric properties of Markov and Lagrange spectra},
\newblock available at http://w3.impa.br/$\sim$gugu.
\bibitem{MY}
{\bf C.G. Moreira and J.C. Yoccoz},
\newblock{\it Stable Intersections of Cantor Sets with Large Hausdorff Dimension},
\newblock Annals of Mathematics {\bf 154} (2001), 45--96.

\bibitem{MY2}
{\bf C.G. Moreira and J.C. Yoccoz},
\newblock{\it Tangences homoclines stables pour des ensembles hyperboliques de grande dimension fractale},
\newblock Annales scientifiques de l'ENS 43, fascicule 1 (2010).

\bibitem{PT}
{\bf J. Palis and F. Takens},
\newblock{\it Hyperbolicity and sensitive chaotic dynamics at homoclinic bifurcations},
\newblock Cambridge studies in advanced mathematics, Cambridge (1993).

\bibitem{PY}
{\bf J. Palis and J.C. Yoccoz},
\newblock{\it On the Arithmetic Sum of Regular Cantor Sets},
\newblock Annales de l'Inst. Henri Poincaré, Analyse Non Lineaire {\bf 14} (1997), 439--456.

\bibitem{PeSh}
{\bf Y. Peres and P. Shmerkin},
\newblock{\it Resonance between Cantor sets},
\newblock Ergodic Theory \& Dynamical Systems {\bf 29} (2009), 201--221.

\bibitem{R}
{\bf M. Rams},
\newblock{\it Exceptional parameters for iterated function systems with overlaps},
\newblock Period. Math. Hungar. {\bf 37} (1998), no. 1-3, 111--119.

\bibitem{R2}
{\bf M. Rams},
\newblock{\it Packing dimension estimation for exceptional parameters},
\newblock Israel J. Math. {\bf 130} (2002), 125--144.


\end{thebibliography}

\end{document}